\documentclass[11pt]{amsart} \usepackage{latexsym, amssymb, stmaryrd}
\usepackage[T1]{fontenc}

\DeclareTextSymbol{\thh}{T1}{254}

\newtheorem{thm}{Theorem}[section]
\newtheorem{lemma}[thm]{Lemma}
\newtheorem{prop}[thm]{Proposition}
\newtheorem{cor}[thm]{Corollary}

\theoremstyle{definition}
\newtheorem{df}[thm]{Definition}
\newtheorem{rmk}[thm]{Remark}

\newtheorem{fact}[thm]{Fact}

\newtheorem{question}[thm]{Question}

\newcommand{\R}{\mathbb{R}}
\newcommand{\Z}{\mathbb{Z}}

\newcommand{\curly}[1]{\mathcal{#1}}
\newcommand{\A}{\curly{A}}

\newcommand{\C}{\curly{C}}

\newcommand{\E}{\curly{E}}

\newcommand{\G}{\curly{G}}

\newcommand{\m}{\mathcal{M}}
\newcommand{\N}{\mathcal{N}}

\newcommand{\la}{\curly{L}}

\makeatletter

\def\indsym#1#2{%
  \setbox0=\hbox{$\m@th#1x$}%
  \kern\wd0%
  \hbox to 0pt{\hss$\m@th#1\mid$\hbox to 0pt{$\m@th#1^{#2}$}\hss}%
  \lower.9\ht0\hbox to 0pt{\hss$\m@th#1\smile$\hss}%
  \kern\wd0}

\def\nindsym#1#2{%
  \setbox0=\hbox{$\m@th#1x$}%
  \kern\wd0%
  \hbox to 0pt{\hss$\m@th#1\not$\kern1.4\wd0\hss}
  \hbox to 0pt{\hss$\m@th#1\mid$\hbox to 0pt{$\m@th#1^{\,#2}$}\hss}%
  \lower.9\ht0\hbox to 0pt{\hss$\m@th#1\smile$\hss}%
  \kern\wd0}

\def\dotminussym#1#2{%
  \setbox0=\hbox{$\m@th#1-$}%
  \kern.5\wd0%
  \hbox to 0pt{\hss\hbox{$\m@th#1-$}\hss}%
  \raise.6\ht0\hbox to 0pt{\hss$\m@th#1.$\hss}%
  \kern.5\wd0}

\def \r { {\mathbb R} }
\def \<{\langle}
\def \>{\rangle}
\def \n {\mathbb N}

\def \*Z {{{^*}\Z}}

\def \((  {(\!(}
\def \)) {)\!)}

\def \tp{\operatorname{tp}}

\def \int{\operatorname{int}}
\numberwithin{equation}{section}

\def \Aut{\operatorname{Aut}}

\def \k{\mathcal{K}}

\def \u{\mathcal{U}}

\def\R{\mathcal R}
\def \Th{\operatorname{Th}}
\def \KH{\operatorname{KH}}
\def \Inn{\operatorname{Inn}}
\def \AppInn{\operatorname{AppInn}}
\def \vna{\operatorname{vNa}}
\def \tr{\operatorname{tr}}

\allowdisplaybreaks[2]

\begin{document}

\title{Existentially closed II$_1$ factors}
\author{Ilijas Farah, Isaac Goldbring, Bradd Hart, and David Sherman}
\thanks{The authors would like to thank NSERC and the Fields Institute for supporting this work.  Goldbring's work was partially supported by NSF grant DMS-1007144. Farah's work was partially supported by 
 a Velux Visiting Professorship and the Danish Council for Independent Research through Asger 
 T\"ornquist's grant, no. 10-082689/FNU.  Sherman's work was partially supported by NSF grant DMS-1201454.}

\address{Department of Mathematics and Statistics, York University, 4700 Keele Street, North York, Ontario, Canada, M3J 1P3, and Matematicki Institut, Kneza Mihaila 35, Belgrade, Serbia, and
Department of Mathematical Sciences, University of Copenhagen, Universitetsparken 5, 2100 Copenhagen, Denmark}
\email{ifarah@mathstat.yorku.ca}
\urladdr{http://www.math.yorku.ca/~ifarah}

\address {Department of Mathematics, Statistics, and Computer Science, University of Illinois at Chicago, Science and Engineering Offices M/C 249, 851 S. Morgan St., Chicago, IL, 60607-7045}
\email{isaac@math.uic.edu}
\urladdr{http://www.math.uic.edu/~isaac}

\address{Department of Mathematics and Statistics, McMaster University, 1280 Main Street W., Hamilton, Ontario, Canada L8S 4K1}
\email{hartb@mcmaster.ca}
\urladdr{http://www.math.mcmaster.ca/~bradd}

\address{Department of Mathematics, University of Virginia, P. O. Box 400137, Charlottesville, VA 22904-4137}
\email{dsherman@virginia.edu}
\urladdr{http://people.virginia.edu/~des5e}

\begin{abstract}
We examine the properties of existentially closed ($\R^\omega$-embeddable) II$_1$ factors.  In particular, we use the fact that every automorphism of an existentially closed ($\R^\omega$-embeddable) II$_1$ factor is approximately inner to prove that $\Th(\R)$ is not model-complete.  We also show that $\Th(\R)$ is complete for both finite and infinite forcing and use the latter result to prove that there exist continuum many nonisomorphic existentially closed models of $\Th(\R)$.
\end{abstract}

\maketitle

\section{Introduction}

This paper continues the model-theoretic study of tracial von Neumann algebras initiated in \cite{FHS1}, \cite{FHS2}, \cite{FHS3}, and \cite{GHS}.  Our main focus is studying the class of \emph{existentially closed} tracial von Neumann algebras.  Roughly speaking, a tracial von Neumann algebra $M$ is existentially closed if any system of $*$-polynomials with parameters from $M$ that has a solution in an extension of $M$ already has an approximate solution in $M$.  It has been observed by many people that an existentially closed tracial von Neumann algebra must be a McDuff II$_1$ factor; see \cite{GHS} for a proof.  In particular, free group factors and ultraproducts of matrix algebras are not existentially closed.

Since the theory of tracial von Neumann algebras is universally axiomatizable, standard model theory shows that every tracial von Neumann algebra is contained in an existentially closed one.  A natural problem arises:  name a concrete existentially closed II$_1$ factor.  Well, one might guess that the hyperfinite II$_1$ factor $\R$ is existentially closed.  It turns out that if one restricts one's attention to tracial von Neumann algebras that embed into ultrapowers of $\R$, henceforth referred to as $\R^\omega$-embeddable von Neumann algebras, then $\R$ is existentially closed; in model-theoretic terms:  $\R$ is an existentially closed model of its universal theory.  (This observation had been made independently by C. Ward Henson and the fourth named author.)  

Recall that the Connes Embedding Problem (CEP) asks whether every II$_1$ factor is $\R^\omega$-embeddable.  It follows that a positive solution to the CEP implies that $\R$ is an existentially closed II$_1$ factor.  In fact, CEP is equivalent to the statement that $\R$ is an existentially closed II$_1$ factor (see Corollary \ref{eccep}).


In Section \ref{inductive}, we show that $\R$ is an existentially closed model of its universal theory and that all existentially closed $\R^\omega$-embeddable factors have the same $\forall\exists$-theory as $\R$.  We also show that the only possible complete $\forall\exists$-axiomatizable theory of $\R^\omega$-embeddable II$_1$ factors is $\Th(\R)$.  In particular, theories of free group factors or ultraproducts of matrix algebras are not $\forall\exists$-axiomatizable.

In Section \ref{auto}, we show that every automorphism $\alpha$ of an existentially closed II$_1$ factor $M$ is \emph{approximately inner}, meaning that, for every finite subset $F$ of $M$ and every $\epsilon>0$, there is a unitary $u$ from $M$ such that $\|\alpha(x)-uxu^*\|_2<\epsilon$ for all $x\in F$.  We use this result to show that $\Th(\R)$ is \emph{not} model-complete, meaning that not every embedding between models of $\Th(\R)$ is elementary.  (It was shown in \cite{GHS} that a positive solution to CEP implied that $\Th(\R)$ was not model-complete.)  As a consequence, we deduce that the class of existentially closed $\R^\omega$-embeddable II$_1$ factors is not an axiomatizable class (a result that was shown to hold for the potentially larger class of existentially closed II$_1$ factors in \cite{GHS}).

In Section \ref{amalg}, we show that every existentially closed $\R^\omega$-embeddable II$_1$ factor is a \emph{strong amalgamation base}:  if $M_0$ is an existentially closed $\R^\omega$-embeddable II$_1$ factor and $f_i:M_0\to M_i$, $i=1,2$, are embeddings into $\R^\omega$-embeddable II$_1$ factors $M_1$ and $M_2$, then there is an $\R^\omega$-embeddable II$_1$ factor $N$ and embeddings $g_i:M_i\to N$ such that $g_1\circ f_1=g_2\circ f_2$ and $g_1(M_1)\cap g_2(M_2)=g_1(f_1(M_0))$.  Until this point, the best known amalgamation result appeared in \cite{BDJ}, where it is shown that the amalgamated free product $M_1*_\R M_2$ is $\R^\omega$-embeddable if both $M_1$ and $M_2$ are $\R^\omega$-embeddable.  Notice our result is not a generalization of the result in \cite{BDJ} as we do not claim that our amalgam is the amalgamated free product; on the other hand, our result applies to continuum many II$_1$ factors instead of applying solely to $\R$.

In the final two sections, we study subclasses of the class of existentially closed II$_1$ factors that are even more generic.  These factors are obtained by \emph{model-theoretic forcing}.  It is shown that $\Th(\R)$ is complete for both of these notions of forcing, meaning that $\Th(\R)$ is the theory of the ``generic'' $\R^\omega$-embeddable factors obtained from these notions of forcing.  As a consequence, we can infer that there are continuum many nonisomorphic models of $\Th(\R)$ that are existentially closed.

Throughout this paper, by a \emph{tracial von Neumann algebra} we mean a pair $(A,\tr)$, where $A$ is a von Neumann algebra and $\tr$ is a fixed, faithful, normal tracial state, although we often suppress mention of the tracial state and simply refer to $A$ as a tracial von Neumann algebra if there is no fear of confusion.  Given a tracial von Neumann algebra $A$, we often consider the \emph{$2$-norm} on $A$ given by $\|x\|_2:=\sqrt{\tr(x^*x)}$.  Given a $*$-monomial $p(\vec x,\vec x^*)$ in the variables $\vec x$ and their adjoints and a tuple $\vec a$ from a tracial von Neumann algebra $A$, the quantity $\tr (p(\vec a,\vec a^*))$ is referred to as a \emph{moment of $\vec a$} and the (total) degree of $p$ is called the \emph{order} of $\tr(p(\vec a,\vec a^*))$.

We will work in the setting of continuous model theory.  We refer the reader to \cite{FHS2} for a rapid introduction to this setting, where it is also explained how to treat von Neumann algebras as metric structures.  However, for the sake of completeness, we recall some basic notions from continuous model theory and discuss the notion of existentially closed structures in the next subsection.  

\subsection{Existentially closed structures}

Fix a continuous language $\la$ (e.g. the language for tracial von Neumann algebras).  The set of $\la$-terms is the smallest set of expressions containing the constant symbols and variables and closed under the function symbols.  For example, in the language for tracial von Neumann algebras, these would be $*$-polynomials.  Atomic $\la$-formulae are expressions of the form $R(t_1,\ldots,t_n)$, where $R$ is a predicate symbol and each $t_i$ is a term.  Continuing with the example of tracial von Neumann algebras, $\tr(p(\bar x,\bar x^*))$ is an atomic formula, where $p(\bar x,\bar x^*)$ is a $*$-polynomial, as is $\|p(\bar x,\bar x^*)\|_2$.  The set of quantifier-free $\la$-formulae is the set of $\la$-formulae obtained from the atomic $\la$-formulae by using continuous functions $f:\r^n\to \r$ as connectives.  If one, in addition, allows the use of ``quantifiers'' $\sup$ and $\inf$, then one arrives at the set of $\la$-formulae.  

Returning to tracial von Neumann algebras again, the expression $\varphi$ given by $\sup_y(\|xy-yx\|_2+\|xx^*-1\|_2)$ is a formula; we may write $\varphi$ as $\varphi(x)$ to indicate that the variable $x$ is \emph{free} in $\varphi$.  If $M$ is a tracial von Neumann algebra and $a\in M$, then plugging $a$ in for $x$ in $\varphi$ returns a real number $\varphi(a)^M$.  The \emph{condition} $\varphi(x)=0$ asserts that $x$ is a unitary element of the center of $M$.  Notice that the $\sup$ equalling $0$ tells us that $x$ commutes with all elements of $M$; it is for this reason we sometimes think of $\sup$ as a universal quantifier, a fact we elaborate on below.

Formulae with no free variables are called \emph{sentences} and conditions $\sigma=0$ where $\sigma$ is a sentence are called \emph{closed conditions}.  Closed conditions actually assert something.  For example, if $\tau$ is the sentence $\sup_x\sup_y \|xy-yx\|_2$, then the closed condition $\tau=0$ holds in a tracial von Neumann algebra $M$ if and only if $M$ is abelian.  It is important to note that, given a sentence $\sigma$, there is a compact interval $I\subseteq \r$ such that $\sigma^M\in I$ for every $\la$-structure $M$; if $I$ is contained in the set of nonnegative real numbers, then we call $\sigma$ a \emph{nonnegative sentence}.


Suppose that $M$ and $N$ are $\la$-structures and $i:M\to N$ is an embedding, that is, an injective map that preserves all the interpretations of symbols in $\la$.  We say that $i$ is an \emph{elementary embedding} if, for every $\la$-formula $\varphi(x)$ and every tuple $a$ from $M$, we have $\varphi(a)^M=\varphi(i(a))^N$.  If we relax the previous definition to only hold for formulae of the form $\inf_{\vec x}\varphi(\vec x)$ with $\varphi(\vec x)$ quantifier-free, we say that $i$ is an \emph{existential embedding}.  If $M$ is a substructure of $N$, henceforth denoted $M\subseteq N$, and the inclusion map $i:M\to N$ is an elementary embedding, we say that $M$ is an \emph{elementary substructure} of $N$ and write $M\preceq N$. 

An $\la$-structure $M$ \emph{models the closed condition} $\sigma=0$, denoted $M\models \sigma=0$, if $\sigma^M=0$.  If $T$ is a set of closed conditions, then $M\models T$ if it models all of the conditions in $T$; we let $\operatorname{Mod}(T)$ denote the class of all models of $T$.  If $T$ is a set of closed conditions and $\sigma$ is a sentence, we say that $T$ logically implies the condition $\sigma=0$, denoted $T\models \sigma=0$, if every model of $T$ is also a model of $\sigma=0$.

In this paper, an \emph{$\la$-theory} is a collection $T$ of closed conditions closed under logical implication, meaning if $T\models \sigma=0$, then $\sigma=0$ belongs to $T$.  If $M$ is an $\la$-structure, then the \emph{theory of $M$} is the theory $$\Th(M):=\{\sigma=0 \ : \ \sigma^M=0\}.$$  A \emph{complete theory} is a theory of the form $\Th(M)$ for some $M$.  If $\Th(M)=\Th(N)$, we say that $M$ and $N$ are \emph{elementarily equivalent} and write $M\equiv N$.  Given any class $\k$ of $\la$-structures, we define the theory of $\k$ to be the theory $$\Th(\k):=\{\sigma=0 \ : \ \sigma^M=0 \text{ for all }M\in \k\}=\bigcap_{M\in \k}\Th(M).$$

Suppose that $M$ and $N$ are $\la$-structures.  In analogy with classical logic, it is remarked in \cite[Section 6]{FHS3} that $M$ embeds into an ultrapower of $N$ if and only if $\sigma^M\leq \sigma^N$ for every $\sup$-sentence $\sigma$, that is, when $\sigma$ is of the form $\sigma=\sup_{\vec x}\varphi(\vec x)$ with $\varphi$ quantifier-free.  Notice that this latter property is equivalent to the property that $\sigma^N=0\Rightarrow \sigma^M=0$ for every nonnegative $\sup$-sentence $\sigma$.  Indeed, while one direction is clear, the other direction follows from the fact that, if $\sigma^N=r$, then $\sup_{\vec x}(\max(\varphi(\vec x)-r,0))$ is nonnegative and has value $0$ in $N$.

With the preceding paragraph in mind, given a theory $T$, we let $T_\forall$ denote the subset of $T$ containing only those conditions $\sigma=0$ for which $\sigma$ is a nonnegative $\sup$-sentence.  Given a nonnegative $\sup$-sentence $\sigma=\sup_{\vec x}\varphi(\vec x)$, the condition $\sigma=0$ asserts that, for all $\vec x$, we have $\varphi(\vec x)=0$; it is for this reason that we may call such a condition a \emph{universal condition} and hence refer to $T_\forall$ as the \emph{universal theory of $T$}.  If $T=\Th(M)$ for some structure $M$, we write $\Th_\forall(M)$ for $T_\forall$.  The content of the previous paragraph may be summarized as:  $M$ embeds into an ultrapower of $N$ if and only if $M\models \Th_\forall(N)$.  More generally, it is readily verified that, given an arbitrary (i.e. perhaps incomplete) theory $T$, we have $M\models T_\forall$ if and only if $M$ embeds into a model of $T$.

Suppose that $\k$ is a class of $\la$-structures.  We say that $\k$ is an \emph{axiomatizable class} if $\k=\operatorname{Mod}(T)$ for some theory $T$.  If there is a theory $T$ such that $\k=\operatorname{Mod}(T_\forall)$, we say that $T$ is \emph{universally axiomatizable}.  By the preceding paragraph, $T$ is universally axiomatizable if and only if a substructure of a model of $T$ is also a model of $T$.  

If $\k$ is a class of $\la$-structures, we say that $\k$ is \emph{inductive} if $\k$ is closed under unions of chains.  If $T$ is an $\la$-theory, we say that $T$ is \emph{inductive} if $\operatorname{Mod}(T)$ is inductive.  In classical logic, a theory is inductive if and only if it is $\forall\exists$-axiomatizable.  In analogy with the preceding paragraphs, the situation in continuous logic is as follows:  given a theory $T$, we let $T_{\forall\exists}$ denote the collection of conditions $\sigma=0$ where $\sigma$ is nonnegative and of the form $\sup_{\vec x}\inf_{\vec y}\varphi(\vec x,\vec y)$ with $\varphi(\vec x,\vec y)$ quantifier-free; notice that such a condition asserts that, for every $\vec x$, there is $\vec y$ such that $\varphi(\vec x,\vec y)$ is (almost) $0$, whence such a condition is morally an $\forall\exists$ assertion.  It is shown in \cite{U} that an axiomatizable class $\k$ is inductive if and only if it is \emph{$\forall\exists$-axiomatizable}, that is, if and only if it is of the form $\operatorname{Mod}(T_{\forall\exists})$ for some theory $T$.  We also say that a theory is $\forall\exists$-axiomatizable if the class of its models is $\forall\exists$-axiomatizable.


In \cite{FHS2}, it is shown that the class of tracial von Neumann algebras is a universally axiomatizable class and the subclass of II$_1$ factors is an $\forall\exists$-axiomatizable class.  By the above fact concerning $\Th_\forall(\R)$, a tracial von Neumann algebra $M$ models $\Th_\forall(\R)$ if and only if it embeds in an ultrapower of $\R$.

\begin{df}
Fix a class $\k$ of $\la$-structures.  For $M,N\in \k$ with $M\subseteq N$, we say that $M$ is \emph{existentially closed in $N$} if, for any quantifier-free formula $\varphi(x,y)$, any $a\in M$, we have $$\inf_{c\in M}\varphi(c,a)^M=\inf_{b\in N}\varphi(b,a)^N.$$ We say that $M\in \k$ is \emph{existentially closed (e.c.) for $\k$} if $M$ is existentially closed in $N$ for every $N\in \k$ with $M\subseteq N$. If $\k$ is the class of models of some theory $T$, we call an e.c. member of $\k$ an \emph{e.c. model of $T$}.
\end{df}

It is well known that if $M$ and $N$ are models of $T_\forall$ with $M$ existentially closed in $N$ and $N$ an e.c. model of $T_\forall$, then $M$ is an e.c. model for $T_\forall$; see the proof of Lemma 6.30 in \cite{Poizat}.  In particular, an elementary substructure of an e.c. model of $T_\forall$ is an e.c. model of $T_\forall$.

It is also well known that if $\k$ is an inductive class, then any member of $\k$ is contained in an e.c. member of $\k$.  If, in addition, the class $\k$ is axiomatizable and the language is, say, countable, then, by Downward L\"owenheim-Skolem, any member of $\k$ is contained in an e.c. member of $\k$ of the same density character.

We say that the class $\k$ is \emph{model-complete} if, for any $M,N\in \k$ with $M\subseteq N$, we have $M\preceq N$.  If $\k$ is a model-complete axiomatizable class, say $\k=\operatorname{Mod}(T)$, we also say that $T$ is model-complete.  Robinson's test for model-completeness states that $T$ is model-complete if and only if every embedding between models of $T$ is existential.

We say that a class $\mathcal C$ of structures is \emph{model-consistent with $\k$} if every element of $\k$ is contained in an element of $\mathcal C$.  For example, the subclass of e.c. elements of $\k$ is model-consistent with $\k$ (by the above remarks).

For any theory $T$, we let $\mathcal E_T$ denote the class of existentially closed models of $T_\forall$.  In general, $\mathcal E_T$ need not be axiomatizable.  If $\mathcal E_T$ is axiomatizable, say $\mathcal E_T=\operatorname{Mod}(T')$, we call $T'$ the \emph{model-companion of $T$}.  (The use of the definite article ``the'' is justified as the model-companion of a theory, if it exists, is unique up to logical equivalence.)  Note that the model-companion of $T$ is necessarily model-complete by Robinson's test.  Conversely, if $T$ is a model-complete theory, then all models of $T$ are e.c. models of $T_\forall$ and $T$ is the model-companion of $T_\forall$.

In \cite{GHS} it is shown that $T_{\vna}$ does not have a model companion, that is, the class of existentially closed tracial von Neumann algebras is not axiomatizable.  It was also shown there that a positive solution to CEP implies that the same conclusion remains true for the class of $\R^\omega$-embeddable factors.  The main result of Section 3 of this paper shows that we may remove the CEP assumption from this latter result.

Suppose that $T$ is a universally axiomatizable theory.  In \cite{U}, it is shown that ``most'' elements of $T$ are e.c. in a sense we now explain. Let $X$ denote the space of all models $M$ of $T$ equipped with a distinguished countable dense subset $M_0\subseteq M$, enumerated as $(m_i:i<\omega)$.  We can define a topology on $X$ by declaring sets of the form $$\{M \in X \ : \ \varphi^M(m_{i_1},\ldots,m_{i_n})<\epsilon\}$$ to be basic open sets, where $\varphi$ is a quantifier-free formula, $i_1,\ldots,i_n\in \n$ and $\epsilon\in \r^{>0}\cup \{+\infty\}$.  In this way, $X$ becomes a Polish space.  It is a consequence of results from \cite{U} that the set of e.c. elements of $X$ is dense in $X$ and any reasonable probability measure on $X$ gives the set of e.c. models full measure.

In our applications to von Neumann algebras, the universal theory $T$ at hand is either the theory of tracial von Neumann algebras $T_{\vna}$ or the theory of $\R^\omega$-embeddable tracial von Neumann algebras $\Th_\forall(\R)$. In this context, we see that any ($\R^\omega$-embeddable) tracial von Neumann algebra embeds into an existentially closed ($\R^\omega$-embeddable) tracial von Neumann algebra.  Recall from the introduction that e.c. tracial von Neumann algebras are McDuff II$_1$ factors; the same proof also shows that e.c. models of $\Th_\forall(\R)$ are McDuff II$_1$ factors.

\begin{fact}\label{NPS}{\cite{NPS}}
There is a family $(M_\alpha)_{\alpha<2^{\aleph_0}}$ of $\R^\omega$-embeddable II$_1$ factors such that, for any II$_1$ factor $M$, at most countably many of the $M_\alpha$'s embed into $M$.
\end{fact}

Consequently, we have:

\begin{cor}
There are $2^{\aleph_0}$ many nonisomorphic existentially closed ($\R^\omega$-embeddable) tracial von Neumann algebras.
\end{cor}

If we knew that $\Th(\R)$ was inductive, then it would follow that there are continuum many nonisomorphic e.c. models of $\Th(\R)$.  Nevertheless, we will be able to derive this conclusion from our work on infinitely generic structures in Section \ref{infgenericsection}.

\section{Inductive Theories of II$_1$ factors}\label{inductive}

Throughout this paper, $\u$ denotes a nonprincipal ultrafilter on some index set; if $M$ is a tracial von Neumann algebra, then $M^\u$ denotes the corresponding (tracial) ultrapower of $M$.  We frequently make use of the fact that every embedding $\R\to \R^\u$ is elementary (as every such embedding is unitarily conjugate to the diagonal embedding; this is an easy direction of the main result of 
 \cite{Jung}).
 
 \emph{From now on, we assume that all sentences under consideration are nonnegative.}  Such an assumption poses no loss of generality when studying existentially closed models (by just adding a suitable real to a formula if necessary) and is required when studying questions of universal and $\forall\exists$-axiomatizability.

The following observation is crucial.  This observation was also independently made by C. Ward Henson and the fourth named author.

\begin{lemma}\label{Rec}
$\R$ is an e.c. model of $\Th_\forall(\R)$.
\end{lemma}

\begin{proof}
Suppose $\varphi(x,y)$ is quantifier-free, $a\in \R$ and $\R\subseteq M$.  Fix an embedding $f:M \to\R^\u$ of $M$ into an ultrapower of $\R$.  Note then that $(\inf_x\varphi(x,a))^\R\geq (\inf_x\varphi(x,a))^M\geq (\inf_x(\varphi(x,fa))^{\R^\u}$.  Since $f|\R$ is elementary, the ends of the double inequality are equal, whence $(\inf_x\varphi(x,a))^\R=(\inf_x\varphi(x,a))^M$.
\end{proof}

In Proposition \ref{Rgeneric}, we will see that $\R$ is even more generic than just being existentially closed.

\begin{cor}\label{eccep}
$\R$ is an e.c. model of $T_{\vna}$ if and only if CEP has a positive solution.
\end{cor}

\begin{proof}
The ``if'' direction follows from Lemma \ref{Rec}.  For the converse, suppose that $M$ is a II$_1$ factor and that $\sigma=0$ belongs to $\Th_\forall(\R)$.  Without loss of generality, suppose that $\sigma$ has value bounded by $1$ in all structures.  Since $\tau:=\max(1-\sigma,0)$ is (equivalent to) a sentence of the form $\inf_x \varphi(x)$, if $\R$ were an e.c. model of $T_{\vna}$, we would have that $\tau^\R=\tau^M$, whence $\sigma^\R=\sigma^M$ and $M\models \Th_\forall(\R)$.
\end{proof}

We now turn to $\forall\exists$-theories of II$_1$ factors.

\begin{lemma}\label{standard}
Suppose that $M\models \Th_\forall(\R)$.
\begin{enumerate}
\item If $M$ is a II$_1$ factor (or, more generally, contains a copy of $\R$ as a substructure), then $\Th_{\forall\exists}(M)\subseteq \Th_{\forall\exists}(\R)$.
\item If $M$ is an e.c. model of $\Th_\forall(\R)$, then $\Th_{\forall \exists}(M)\supseteq\Th_{\forall \exists}(\R)$.
\end{enumerate}
Consequently, if $M$ is an e.c. model of $\Th_\forall(\R)$, then $\Th_{\forall\exists}(M)=\Th_{\forall\exists}(\R)$.
\end{lemma}

\begin{proof}
(1)  Suppose that $\sigma=0$ belongs to $\Th_{\forall\exists}(M)$; write $\sigma=\sup_x\inf_y\varphi(x,y)$.  Fix $a\in \R$.  We have embeddings $i:\R\to M$ and $j:M\to \R^\u$.  Since $(\inf_y \varphi(i(a),y))^M=0$, we have $(\inf_y \varphi(j(i(a)),y))^{\R^\u}=0$.  Since $j\circ i$ is elementary, we have $(\inf_y \varphi(a,y))^\R=0$.  Since $a\in \R$ was arbitrary, we have $\sigma^\R=0$.  

(2) is standard (and holds in complete generality) but we include a proof for the sake of completeness. Suppose $\sigma=\sup_x\inf_y\varphi(x,y)$ and $\sigma^\R=0$.  Since $M$ embeds into an ultrapower $\R^\u$ of $\R$, given $a\in M$, we know that $(\inf_y\varphi(a,y))^M=(\inf_y\varphi(a,y))^{\R^\u}=0$, whence $\sigma^M=0$.

The last statement of the lemma follows from the fact (discussed above) that e.c. models of $\Th_\forall(\R)$ are II$_1$ factors.
\end{proof}

%

\begin{cor}
If there exists a \emph{complete} $\forall \exists$-axiomatizable theory $T'$ of $\R^\omega$-embeddable II$_1$ factors, then $T'=\Th(\R)$.  
\end{cor}

\begin{proof}
Suppose that $M\models \Th_\forall(\R)$ is a II$_1$ factor such that $T':=\Th(M)$ is $\forall\exists$-axiomatizable.  Then by (1) of the previous lemma, $\R\models T'$, whence $T'=\Th(\R)$.
\end{proof}
%

The previous corollary shows that if $M$ is an $\R^\omega$-embeddable II$_1$ factor that is not elementarily equivalent to $\R$, then $\Th(M)$ is not $\forall\exists$-axiomatizable.  In particular, if $M$ is not McDuff (e.g. $M$ is a free group factor or ultraproduct of matrix algebras), then $\Th(M)$ is not $\forall\exists$-axiomatizable.

\begin{question}
Is $\Th(\R)$ $\forall\exists$-axiomatizable?
\end{question}

The previous question has a purely operator-algebraic reformulation in light of the model-theoretic fact that a theory in a countable language is $\forall\exists$-axiomatizable if and only if it is closed under unions of chains of countable models:  if we have a chain 
$$\R_0\subseteq \R_1\subseteq \R_2\subseteq \cdots$$ such that $\R_i^\u\cong \R^\u$ for each $i$, then setting $\R_\infty:=\overline{\bigcup_i \R_i}$, do we have $\R_\infty^\u \cong \R^\u$?

We should remark that in connection with the question of axiomatizability, we do know that $\Th(\R)$ is \emph{not} $\exists\forall$-axiomatizable.  Indeed, let $i:\R\to L(\mathbb F_2)$ and $j:L(\mathbb F_2)\to \R^\u$ be embeddings.  Consider $\sigma:=\inf_x \sup_y\varphi(x,y)$ such that $\sigma^\R=0$.  Fix $\epsilon>0$ and choose $a\in \R$ such that $\sup_y\varphi(a,y)^\R<\epsilon$.  Since $j\circ i$ is elementary, we have $$\sup_y\varphi(i(a),y)^{L(\mathbb F_2)}\leq \sup_y\varphi(j(i(a)),y)^{\R^\u}<\epsilon.$$  It follows that $\sigma^{L(\mathbb F_2)}=0$.  Thus, if $\Th(\R)$ were $\exists\forall$-axiomatizable, we would have that $L(\mathbb F_2)\equiv \R$, a contradiction.

For any theory $T$, recall that $\E_T$ denote the existentially closed models of $T_\forall$.

\begin{cor}
If $T=\Th(\R)$, then $\Th_{\forall\exists}(\E_T)=\Th_{\forall\exists}(\R)$.
\end{cor}

\begin{proof}
Lemma \ref{standard}(2) shows that $\Th_{\forall\exists}(\R)\subseteq \Th_{\forall\exists}(\E_T)$.  Conversely, suppose that $\sigma=0$ belongs to $\Th_{\forall\exists}(\E_T)$.  Since e.c. models of $T$ are II$_1$ factors, we have, by Lemma \ref{standard}(1), that $\sigma^\R=0$.
\end{proof}

\begin{rmk}
If $\Th(\R)$ is $\forall\exists$-axiomatizable, then $\Th(\E_T)=\Th(\R)$.  This would be in contrast to the theory of groups, where there are non-elementarily equivalent e.c. groups. 
\end{rmk}

At this point, it makes sense to introduce companion operators and the Kaiser hull into continuous logic.

\begin{df}
Suppose that we have a mapping $T\mapsto T^*$ on theories.  We say that the mapping is a \emph{companion operator} if, for all theories $T$ and $T'$, we have:
\begin{enumerate}
\item $(T^*)_\forall=T_\forall$
\item $T_\forall=T'_\forall\Rightarrow T^*=(T')^*$.
\item $T_{\forall\exists}\subseteq T^*$.
\end{enumerate}
\end{df}

In what follows, we will see some examples of companion operators.  It is clear that if $T$ has a model companion $T'$, then $T'$ satisfies the conditions of the previous definition. The notion of a companion operator was an attempt to extend the notion of a model companion to an operation that is defined for all theories.
 
\begin{lemma}
Suppose that $T^*$ is a companion of $T$ and $T'$ is an inductive theory such $(T')_\forall=T_\forall$.  Then $T'\subseteq T^*$.
\end{lemma}

\begin{proof}
We have $T^*=(T')^*\supseteq (T')_{\forall\exists}$; since $T'$ is inductive, it follows that $T'\subseteq T^*$.
\end{proof}

\begin{lemma}
If $T_1$ and $T_2$ are both inductive theories with the same universal theory as $T$, then so is $T_1\cup T_2$.
\end{lemma}

\begin{proof}
$T_1\cup T_2$ is clearly inductive.  Suppose $\A\models T_\forall$.  We build a chain
$$\A\subseteq \m_0\subseteq \N_0\subseteq \m_1\subseteq \N_1\subseteq \cdots,$$ where each $\m_i\models T_1$ and $\N_i\models T_2$; we obtain the $\m_i$ and $\N_i$ by using the fact that each of $T_1$ and $T_2$ has the same universal theory as $T$.  Since both $T_1$ and $T_2$ are inductive, it follows that the union of the chain models both $T_1$ and $T_2$.
\end{proof}

Since there is an inductive theory with the same universal theory as $T$, namely $T_{\forall \exists}$, the previous two lemmas imply that there is a minimal companion operator obtained by taking the maximal inductive theory with the same universal theory as $T$, which is the union of all inductive theories with the same universal theory as $T$.  This companion is called the \emph{inductive} or \emph{Kaiser hull} of $T$, denoted $T^{\KH}$.

\begin{prop}

\
\begin{enumerate}
\item $\Th_{\forall\exists}(\E_T)\subseteq T_{\KH}$.
\item $T_{\KH}$ is axiomatized by $\Th_{\forall\exists}(\E_T)$.
\end{enumerate}
\end{prop}

\begin{proof}
The first item follows from the fact that any model of $T_\forall$ can be extended to an e.c. model of $T_\forall$.  For the second item, it suffices to prove that every element $M$ of $\E_T$ models $T^{\KH}_{\forall\exists}$; however, this follows immediately from the fact that $M$ is e.c. and is contained in a model of $T^{\KH}$. 
\end{proof}

\begin{cor}
If $T=\Th(\R)$, then $T^{\KH}$ is axiomatized by $\Th_{\forall\exists}(\R)$.
\end{cor}

\section{Automorphisms of e.c. II$_1$ factors}\label{auto}

Suppose that $M$ is a II$_1$ factor and $\alpha\in \Aut(M)$.  Recall that $\alpha$ is said to be \emph{approximately inner} if, for every finite set $\{x_1,\ldots,x_n\}\subseteq M$ and every $\epsilon>0$, there is $u\in U(M)$ such that $\max_{1\leq i\leq n}\|\alpha(x_i)-ux_iu^*\|_2<\epsilon$.  Let $\Inn(M)$ denote the group of inner automorphisms of $M$ and let $\AppInn(M)$ denote the group of approximately inner automorphisms of $M$.  Then $\AppInn(M)$ is the closure of $\Inn(M)$ in the point-strong topology on $\Aut(M)$.  It is a fact, independently due to Connes \cite{connes} and Sakai \cite{Sakai}, that $\AppInn(M)=\Inn(M)$ if and only if $M$ does \emph{not} have property $(\Gamma)$.
\begin{prop}\label{ecappinn}
Suppose that $M$ is an e.c. model of $T_{\vna}$ (so in particular a II$_1$ factor).  Then $\Inn(M)<\AppInn(M)=\Aut(M)$.
\end{prop}

\begin{proof}
That $\Inn(M)$ is a proper subgroup of $\AppInn(M)$ follows from the fact that an e.c. II$_1$ factor is McDuff, whence has $(\Gamma)$.  For the second equality, suppose that $\alpha\in \Aut(M)$.  Set $N:=M\rtimes_\alpha \Z$.  Fix $x_1,\ldots,x_n\in M$.  Then $$N\models \inf_u\max(d(uu^*,1),d(u^*u,1),\max_{1\leq i\leq n}d(\alpha(x_i),ux_iu^*))=0.$$  Since $M$ is e.c., there is an almost unitary which almost conjugates each $x_i$ to $\alpha(x_i)$.  By functional calculus, we can find an actual unitary that almost conjugates each $x_i$ to $\alpha(x_i)$ with slightly worse error.
\end{proof}

\begin{rmk}
If $\alpha\in \AppInn(M)$, then for any elementary extension $M'$ of $M$ that is $\kappa^+$-saturated, where $\kappa$ is the density character of $M$, there is $u\in U(M')$ such that $\alpha(x)=uxu^*$ for all $x\in M$.
\end{rmk}

If $M\models \Th_\forall(\R)$ and $\alpha\in \Aut(M)$, then $M\rtimes_\alpha \Z\models \Th_\forall(\R)$; see \cite[Proposition 3.4]{A}.  We can thus repeat the proof of Proposition \ref{ecappinn} and conclude the following:  

%

\begin{prop}\label{ecappinnR}
If $M$ is an e.c. model of $\Th_\forall(\R)$, then $$\Inn(M)<\AppInn(M)=\Aut(M).$$  
\end{prop} 

In particular, using Lemma \ref{Rec}, we recover the result of Sakai \cite{Sakai} that $\AppInn(\R)=\Aut(\R)$.  We now aim to prove that $\Th(\R)$ is not model-complete.  First, we need the following proposition.  Recall, for an $\la$-structure $M$ and a tuple $a$ from $M$, the \emph{type of $a$ in $M$}, denoted $\tp^M(a)$, is the set of formulae $\varphi(x)$ such that $\varphi^M(a)=0$.
\begin{prop}\label{equaltype}
Suppose that $\Th(\R)$ is model-complete.  Then for any $\R_1\equiv \R$ and any finite tuples $a,b\in \R_1$ of the same length, we have that $\tp^{\R_1}(a)=\tp^{\R_1}(b)$ if and only if $a$ and $b$ are approximately unitarily conjugate in $\R_1$.
\end{prop}

\begin{proof}
Without loss of generality, we may suppose that $\R_1$ is separable.  Certainly if $a$ and $b$ are approximately unitarily conjugate in $\R_1$, then they are unitarily conjugate in some ultrapower $\R_1^\u$ of $\R_1$, whence they have the same type in $\R_1^\u$, and hence in $\R_1$.  Conversely, suppose that $\tp^{\R_1}(a)=\tp^{\R_1}(b)$.  Go to a strongly $\omega$-homogeneous elementary extension $\R_2$ of $\R_1$ (see \cite[Section 7]{BBHU}).   Then there is $\alpha\in \Aut(\R_2)$ such that $\alpha(a)=b$.  Since $\Th(\R)$ is model-complete, every model of $\Th(\R)$ is existentially closed, whence, by Lemma \ref{ecappinnR}, we have that $a$ and $b$ are approximately unitarily conjugate in $\R_2$, and hence in $\R_1$.
\end{proof}

%

%
%

We will need the following:

\begin{fact}[Jung \cite{Jung}]\label{jung}
Suppose that $M$ is a finitely generated $\R^\omega$-embeddable factor such that, for any two embeddings $i,j:M\to \R^\u$, there is $u\in U(\R^\u)$ such that $i(x)=uj(x)u^*$ for all $x\in M$.  Then $M\cong \R$.
\end{fact}

In order to apply Fact \ref{jung}, we must observe that any separable $\R'\equiv \R$ is finitely generated.  In fact, if $\R'\equiv \R$, then $\R'$ is McDuff, whence singly generated (see \cite[Theorem 1]{behncke} for an even more general statement).
\begin{thm}
$\Th(\R)$ is not model-complete.
\end{thm}

\begin{proof}
Suppose, towards a contradiction, that $\Th(\R)$ is model-complete. Suppose that $\R'\preceq \R^\u$ is separable and not isomorphic to $\R$; this is possible by \cite[Theorem 4.3]{FHS3}.  We show that every embedding $j:\R'\to \R^\u$ is implemented by a unitary, that is, there is $u\in \R^\u$ such that, for every $x\in \R'$, $j(x)=uxu^*$; this will contradict Fact \ref{jung}. Fix a generator $x$ for $\R'$.  By model-completeness, $\tp^{\R^\u}(x)=\tp^{\R^\u}(j(x))$.  Thus, by Proposition \ref{equaltype}, $x$ and $j(x)$ are approximately unitarily conjugate in $\R^\u$; since $\R^\u$ is $\omega_1$-saturated (see \cite[Proposition 4.11]{FHS2}), it follows that $x$ and $j(x)$ are unitarily conjugate in $\R^\u$.  It follows that $j$ is implemented by a unitary, yielding the desired contradiction.
\end{proof}

\begin{cor}
$\Th_\forall(\R)$ does not have a model companion.  Consequently, the e.c. models of $\Th_\forall(\R)$ do not form an axiomatizable class.
\end{cor}

\begin{proof}
The proof of \cite[Proposition 3.2]{GHS} shows that any model complete theory of $\R^\omega$-embeddable II$_1$ factors must be contained in $\Th(\R)$.  Thus, if the model companion of $\Th_\forall(\R)$ existed, we would have that $\Th(\R)$ is model-complete, a contradiction.
\end{proof}

We should remark that the fact that $\Th(\R)$ is not model-complete gives a more elementary proof of \cite[Corollary 3.5]{GHS}, namely that CEP implies that there are no model-complete theories of II$_1$ factors.  Indeed, this proof is a bit simpler than the one given in \cite{GHS} as it does not require us to use the fact that $T_{\vna}$ does not have a model companion, which in turn involves some nontrivial results of Nate Brown from \cite{Brown}.

\begin{cor}
Assume that the continuum hypothesis (CH) holds.  Then for any nonprincipal ultrafilter $\u$ on $\n$, there is an embedding $f:\R^\u\to \R^\u$ that is not existential.
\end{cor}

\begin{proof}
By Robinson's test, there are separable $\R_1,\R_2\models \Th(\R)$ and an embedding $g:\R_1\to \R_2$ that is not existential.  Set $f:=g^\u:\R_1^\u\to \R_2^\u$.  Then $f$ is not existential.  By CH, we have $\R_1^\u\cong \R_2^\u\cong \R^\u$ (see \cite{FHS2}), finishing the proof.
\end{proof}

%

The following corollary uses a standard absoluteness argument from set theory. 
A statement is \emph{arithmetical} if all of its quantifiers range over the set of natural numbers, $\n$. 
Forcing and most standard methods for proving relative consistency with ZFC do not add (or remove) 
elements of  $\n$. Therefore the truth of arithmetical statements is invariant under forcing. 
In particular, if one proves such a statement by using an axiom that can be forced over every model of ZFC (such as the Continuum Hypothesis), then the statement can be proved in ZFC alone
(see e.g., \cite{Fa:Absoluteness} for more examples of absolute statements in analysis).

\begin{cor}
There is $\epsilon>0$ such that, for every $m\in \n^{>0}$, there are tuples $a$ and $b$ from $\R$ whose moments up to order at most $m$ are within $\frac{1}{m}$ of each other and for which there is no unitary $u$ in $\R$ that conjugates $a$ to within $\epsilon$ of $b$ (in $2$-norm).
\end{cor}

\begin{proof} 
We first observe that we may safely assume CH in the proof of the corollary. 
Indeed, the truth of the statement remains unaltered if we instead quantify over some ``nice'' (i.e. definable) countable dense subsets of ${\mathbb Q}$ and $\R$; this modified statement is now arithmetical, whence absolute.

Suppose that the statement of the corollary is false; we show that every embedding $f:\R^\u\to \R^\u$, where $\u$ is an ultrafilter on the natural numbers, is existential, contradicting the previous corollary.  Suppose that $f:\R^\u\to \R^\u$ is an embedding and suppose that $\varphi(x,y)$ is a quantifier-free formula (with single variables for simplicity).  Fix $a=[(a_n)]\in \R^\u$ and set $f(a)=b=[(b_n)]$.  Suppose that $(\inf_x\varphi(x,b))^{\R^\u}=r$.  Fix $\eta>0$.  Fix $I_0\in \u$ such that, for $n\in I_0$, we have $(\inf_x \varphi(x,b_n))^\R\leq r+\eta$.  Let $\epsilon:=\Delta_\varphi(\eta)$ and choose $m$ as in the assumption; here $\Delta_\varphi$ is a modulus of uniform continuity for $\varphi$.  Fix $I_1\subseteq I_0$ such that, for $n\in I_1$, we have $a_n$ and $b_n$ have moments up to order $m$ that agree to within $\frac{1}{m}$.  (This is possible because $a$ and $b$ have the same moments.)  For $n\in I_1$, we have unitaries $u_n\in \R$ such that $|u_na_nu_n^*-b_n|<\frac{1}{m}$.  In that case, we get $\inf_x(\varphi(x,a_n))\leq r+2\eta$ for $n\in I_1$.  It follows that $\inf_x(\varphi(x,a))^{\R^\u}\leq r$.
\end{proof}

\section{E.c. models and strong amalgamation bases}\label{amalg}

Until further notice, we let $\la$ be a continuous signature and $\k$ a class of $\la$-structures.

\begin{df}
We say that $A\in \k$ is an \emph{amalgamation base for $\k$} if whenever $B,C\in \k$ both contain $A$, then there is $D\in \k$ and embeddings $f:B\to D$ and $g:C\to D$ such that $f|A=g|A$.  If, in addition, we can always find $D$, $f$, and $g$ such that $f(B)\cap g(C)=f(A)$, we call $A$ a \emph{strong amalgamation base for $\k$}.
\end{df}

If $\k$ is the class of tracial von Neumann algebras, then, by virture of the amalgamated free product construction, every element of $\k$ is an amalgamation base.

For any $\la$-structure $A$, we let $\la_A$ denote the language $\la$ where new constant symbols $c_a$ are added for elements $a\in A$.  We let $D(A)$ denote the atomic diagram of $A$, that is, the set of closed $\la_A$-conditions ``$\sigma(a)=0$,'' where $\sigma(x)$ is a quantifier-free formula, $a$ is a tuple from $A$, and $\sigma(a)^A=0$.  As in classical logic, if $B$ is an $\la_A$-structure that satisfies $D(A)$, then the map sending $a$ to the interpretation of the constant naming $a$ in $B$ is an embedding of $\la$-structures.  

We also let $D^+(A)$ denote the set of all closed conditions ``$\sigma(a)\leq \frac{1}{k}$'', where $\sigma(a)=0$ belongs to $D(A)$ and $k\in \n^{>0}$.  Observe that an $\la_A$ structure satisfies $D(A)$ if and only if it satisfies $D^+(A)$.

The following is the continuous logic analog of a classical model-theoretic fact (see \cite[Theorem 3.2.7]{H}, although for some reason there it is assumed that $T$ is $\forall\exists$-axiomatizable, which is surely unnecessary).
\begin{prop}
Suppose that $T$ is an $\la$-theory and $A$ is an e.c. model of $T$.  Then $A$ is a strong amalgamation base for the models of $T$.
\end{prop}

\begin{proof}
Suppose that $B,C\models T$ both contain $A$.  Without loss of generality, $B\cap C=A$.  For $c\in C\setminus A$, set $\delta_c:=d(c,A)>0$.  It suffices to show that the following set of $\la_{B\cup C}$-conditions is satisfiable:
$$T\cup D^+(B)\cup D(C)\cup \{d(b,c)\geq \delta_c \ | \ b\in B\setminus A, c\in C\setminus A\}.$$  Suppose that this is not the case.  Then there is $k\in \n^{>0}$, $\vec b=(b_1,\ldots,b_n)$ from $B\setminus A$, a quantifier-free formula $\chi(\vec b,\vec d)$, where $\vec d\in A$ and $\chi^B(\vec b,\vec d)=0$, and $c_1,\ldots,c_n$ from $C\setminus A$ such that $$T\cup \{\chi(\vec b,\vec d)\leq \frac{1}{k}\}\cup D(C)\cup \{d(b_i,c_i)\geq \delta_{c_i} \ | \ i=1,\ldots,n\}$$ is unsatisfiable.  Consequently, the set of $\la_C$-conditions
$$T\cup \{\chi(\vec x,\vec d)\leq \frac{1}{k}\}\cup D(C)\cup \{d(x_i,c_i)\geq \delta_{c_i} \ | \ i=1,\ldots,n\}$$ is unsatisfiable.  Since $A$ is e.c., there is $\vec a\in A$ such that $\chi^A(\vec a,\vec d)\leq \frac{1}{k}$, whence $\chi^C(\vec a,\vec d)\leq \frac{1}{k}$.  Consequently, there is $i\in \{1,\ldots,n\}$ such that $d(a_i,c_i)<\delta_{c_i}$, a contradiction.
\end{proof}

Observe in the previous proof that we could have replaced $D(C)$ by the full elementary diagram of $C$, whence we can always assume that the amalgam is an elementary extension of $C$.  Also observe that, by Downward L\"owenheim-Skolem, we can ensure that the amalgam has density character equal to the maximum of the density characters of $B$ and $C$.

\begin{cor}
Any e.c. $\R^\omega$-embeddable  von Neumann algebra is a strong amalgamation base for the class of $\R^\omega$-embeddable von Neumann algebras.  
\end{cor}

We should compare this result with the difficult result of \cite{BDJ} that if $M_i$ are $\R^\omega$-embeddable II$_1$ factors for $i=1,2$, then the amalgamated free product $M_1*_\R M_2$ is also $\R^\omega$-embeddable.  This is the best such result known in the sense that if one replaces $\R$ by another $\R^\omega$-embeddable tracial von Neumann algebra, then it is unknown whether or not the amalgamated free product is $\R^\omega$-embeddable.

\begin{question}
Is every model of $\Th_\forall(\R)$ an amalgamation base?
\end{question}

\section{Infinitely generic structures}\label{infgenericsection}

In this section, we assume that $\k$ is an inductive class of $\la$-structures.  We will prove the existence of a very natural subclass of $\k$, the so-called \emph{infinitely generic} elements of $\k$, which can be characterized as the unique maximal subclass of $\k$ that is model-complete and model-consistent with $\k$.  These structures will turn out to be existentially closed elements of $\k$.  Our treatment of infinitely generic structures in continuous logic is inspired by the classical treatment of this topic presented in \cite{Hirsh}.  For the sake of simplicity, we work in the bounded continuous logic of \cite{BBHU}, where all predicates and formulae take values in $[0,1]$ (although we apply the general theory to the unbounded case of tracial von Neumann algebras).

We arrive at the class of infinitely generic structures via \emph{infinite forcing}.  For $M\in \k$, $\sigma$ a restricted $\la_M$ sentence in prenex normal form (see \cite[Section 6]{BBHU}), $\bowtie \in \{<,\leq, >, \geq\}$, and $r\in [0,1]$, we define the relations $M\Vdash \sigma\bowtie r$ recursively on the complexity of $\sigma$:
\begin{itemize}
\item If $\sigma$ is quantifier-free, then $M\Vdash \sigma \bowtie r$ iff $\sigma^M\bowtie r$.
\item Suppose that $\sigma=\inf_x \varphi(x)$.  Then:
\begin{itemize}
\item $M\Vdash \sigma<r$ iff there is $a\in M$ such that $M\Vdash \varphi(a)<r$.
\item $M\Vdash \sigma\leq r$ iff $M\Vdash \sigma<r'$ for every $r'>r$.
\item $M\Vdash \sigma\geq r$ iff there does not exist $N\in \k$ with $N\supseteq M$ and  $a\in N$ such that $N\Vdash \varphi(a)< r$.
\item $M\Vdash \sigma>r$ iff $M\Vdash \sigma\geq r'$ for some $r'>r$.
\end{itemize}
\item Suppose that $\sigma=\sup_x\varphi(x)$.  Then:
\begin{itemize}
\item $M\Vdash \sigma\leq r$ iff for there does not exist $N\in \k$ with $N\supseteq M$ and $a\in N$ such that $N\Vdash \varphi(a)>r$.
\item $M\Vdash \sigma< r$ iff $M\Vdash \sigma\leq r'$ for some $r'<r$.
\item $M\Vdash \sigma> r$ iff there is $a\in M$ such that $M\Vdash \varphi(a)>r$.
\item $M\Vdash \sigma\geq r$ iff $M\Vdash \sigma> r'$ for all $r'<r$.
\end{itemize}
\end{itemize}

The next three lemmas are routine and are left to the reader.

\begin{lemma}
Suppose that $M\in \k$, $\sigma$ is a restricted $\la_M$ sentence in prenex normal form and $r,s\in [0,1]$ are such that $r<s$.  Then:
\begin{enumerate}
\item If $\bowtie\in \{<,\leq\}$ and $M\Vdash \sigma\bowtie r$, then $M\Vdash \sigma \bowtie s$.
\item If $\bowtie \in \{>,\geq\}$ and $M\Vdash \sigma\bowtie s$, then $M\Vdash \sigma\bowtie r$.
\end{enumerate}
\end{lemma}

\begin{lemma}
Suppose that $M\in \k$, $\sigma$ is a restricted $\la_M$ sentence in prenex normal form and $r\in [0,1]$.  
\begin{enumerate}
\item If $M\Vdash \sigma< r$, then $M\Vdash \sigma\leq r$.
\item If $M\Vdash \sigma> r$, then $M\Vdash \sigma\geq r$. 
\end{enumerate}
\end{lemma}

\begin{lemma}\label{coherency}
Suppose that $M\in \k$, $\sigma$ is a restricted $\la_M$ sentence in prenex normal form, and $r,s\in [0,1]$.  If $\bowtie \in \{<,\leq\}$ and $\bowtie'\in \{>,\geq\}$ are such that $M\Vdash \sigma\bowtie r$ and $M\Vdash \sigma \bowtie' s$, then $s\leq r$.
\end{lemma}

\begin{df}
Suppose that $M\in \k$ and $\sigma$ is a restricted $\la_M$ sentence in prenex normal form.  We then define:
\begin{itemize}
\item $V^M(\sigma):=\inf\{r \ : \ M\Vdash \sigma<r\}=\inf\{r \ : \ M\Vdash \sigma\leq r\}$.
\item $v^M(\sigma):=\sup\{r \ : \ M\Vdash \sigma> r\}=\sup\{r \ : \ M\Vdash \sigma\geq r\}$.
\end{itemize}
We refer to $V^M(\sigma)$ and $v^M(\sigma)$ as the \emph{upper and lower forcing values of $\varphi$ in $M$}.
\end{df}

%
%

By Lemma \ref{coherency}, we see that $v^M(\sigma)\leq V^M(\sigma)$ for any restricted $\la_M$-sentence $\sigma$.

\begin{lemma}\label{compat}
Suppose that $M,N\in \k$ are such that $M\subseteq N$ and $\sigma$ is a restricted $\la_M$ sentence in prenex normal form.  Then $v^M(\sigma)\leq v^N(\sigma)\leq V^N(\sigma)\leq V^M(\sigma)$.
\end{lemma}

\begin{proof}
The proof is by induction on complexity of $\sigma$, the result being obvious for $\sigma$ quantifier-free.  Suppose first that $\sigma=\inf_x \varphi(x)$.  Suppose that $M\Vdash \sigma \geq r$.  If $N\nVdash \sigma\geq r$, then there is $N'\in \k$, $N'\supseteq N$ and $a\in N'$ such that $N'\Vdash \varphi(a)<r$.  Since $M\subseteq N'$, we have $M\nVdash \sigma\geq r$.  It follows that $v^M(\sigma)\leq v^N(\sigma)$.  Now suppose that $M\Vdash \sigma<r$, so there is $a\in M$ such that $M\Vdash \varphi(a)<r$.  By induction, we have $V^N(\varphi(a))\leq r$.  Fix $\epsilon>0$.  Then $N\Vdash \varphi(a)<r+\epsilon$, whence $N\Vdash \sigma<r+\epsilon$.  Since $\epsilon>0$ was arbitrary, we have $V^N(\sigma)\leq r$.  It follows that $V^N(\sigma)\leq V^M(\sigma)$.

The proof for the case $\sigma=\sup_x\varphi(x)$ is similar.
\end{proof}

\begin{df}
We say that $M\in \k$ is \emph{infinitely generic} if, for every restricted $\la_M$ sentence $\sigma$ in prenex normal form, we have $v^M(\sigma)=V^M(\sigma)$.
\end{df}

\begin{prop}\label{contain}
For every $M\in \k$, there is $N\in \k$ with $M\subseteq N$ such that $N$ is infinitely generic.
\end{prop}

\begin{proof}
Suppose that $\sigma$ is a restricted $\la_M$ sentence in prenex normal form.  We seek to find $N\in \k$ with $N\supseteq M$ such that $N$ is generic for $\sigma$.  Since any extension of $N$ in $\k$ remains generic for $\sigma$, we can iterate this process to find an extension of $M$ in $\k$ that is generic for every restricted $\la_M$ sentence in prenex normal form.  We can then iterate this procedure $\omega$ many times to get a generic extension of $M$. 

If $M$ is generic for $\sigma$, we do nothing.  Otherwise, we have $v^M(\sigma)<V^M(\sigma)$.  Note that $\sigma$ cannot be quantifier-free.  Suppose first that $\sigma=\inf_x\varphi(x)$.  Set $r$ to be the midpoint of $(v^M(\sigma),V^M(\sigma))$.  Since $M\nVdash \sigma\geq r$, we have $N_0\in \k$ with $M\subseteq N_0$ and $a\in N_0$ such that $N\Vdash \varphi(a)<r$.  It follows that $V^{N_0}(\sigma)-v^{N_0}(\sigma)\leq \frac{1}{2}(V^M(\sigma)-v^M(\sigma))$.  If $N_0$ is generic for $\sigma$, then we are done.  Otherwise, by the same argument, there is $N_1\in \sigma$ with $N_0\subseteq N_1$ and $V^{N_1}-v^{N_1}(\sigma)\leq \frac{1}{2}(V^{N_0}(\sigma)-v^{N_0}(\sigma))$.  If in this process we ever reach a generic for $\sigma$ extension of $M$, then we are done.  Otherwise, $N:=\overline{\bigcup_i N_i}$ is a generic for $\sigma$ extension of $M$.

Now suppose that $\sigma=\sup_x\varphi(x)$ and let $r$ be as in the previous paragraph.  Since $M\nVdash \sigma\leq r$, there is $N\in \k$ with $M\subseteq N$ and $a\in N$ such that $N\vDash \varphi(a)>r$.  Now proceed as in the previous paragraph.  
\end{proof}

The following characterization of infinitely generic structures relating forcing and truth is crucial.

\begin{prop}
Suppose that $M\in \k$.  Then $M$ is infinitely generic if and only if, for every restricted $\la_M$-sentence $\sigma$, every $r\in [0,1]$ and every $\bowtie \in \{<,\leq,>,\geq\}$, we have
$$M\Vdash \sigma\bowtie r\Leftrightarrow \sigma^M\bowtie r. \quad (\dagger)$$
\end{prop}

\begin{proof}
We treat the ``if'' direction first.  Fix a restricted $\la_M$-sentence $\sigma$ and suppose, towards a contradiction, that $v^M(\sigma)<V^M(\sigma)$.  Fix $r\in (v^M(\sigma),V^M(\sigma))$.  Since $M\nVdash \sigma\geq r$, by $(\dagger)$, we have $\sigma^M<r$.  By $(\dagger)$ again, we see that $M\Vdash \sigma<r$, contradicting $r<V^M(\sigma)$.

We now prove the ``only if'' direction by induction on complexity of $\sigma$.  As usual, the quantifier-free case is trivial and we only treat the case $\sigma=\inf_x\varphi(x)$.  The equivalence in $(\dagger)$ is clear when $\bowtie\in \{<,\leq\}$.  To finish, it suffices to prove that $(\dagger)$ holds for $\bowtie$ equalling $\geq$.  Suppose that $M\Vdash \sigma \geq r$ and yet $\sigma^M<r$.  Then by induction we have that $M\Vdash \varphi(a)<r$ for some $a\in M$, a contradiction.  If $M\nVdash \sigma\geq r$, then $V^M(\sigma)=v^M(\sigma)<r$, whence $M\Vdash \sigma<r$ and hence $\sigma^M<r$.
\end{proof}

Let $\G$ denote the collection of infinitely generic members of $\k$.

\begin{cor}\label{forcetruth}
If $M\in \G$, then for every restricted $\la_M$ sentence $\sigma$, we have $v^M(\sigma)=V^M(\sigma)=\sigma^M$. 
\end{cor}

\begin{prop}
If $M,N\in \G$ and $M\subseteq N$, then $M\preceq N$.
\end{prop}

\begin{proof}
If $\sigma$ is a restricted $\la_M$ sentence in prenex normal form, then $\sigma^M=\sigma^N$ by Lemma \ref{compat} and Corollary \ref{forcetruth}.  It remains to notice that the restricted formulae are dense in the set of all formulae.
\end{proof}


\begin{prop}
If $M\in \G$, then $M$ is e.c. for $\k$.
\end{prop}

\begin{proof}
It is enough to check the condition for the case that $\varphi$ is restricted quantifier-free.  In that case, suppose $b\in M$ and $N\in \k$ is such that $M\subseteq N$.  Take $N'\in \G$ such that $N\subseteq N'$.  Observe that $$(\inf_x \varphi(x,b))^{N'}\leq (\inf_x \varphi(x,b))^N\leq (\inf_x\varphi(x,b))^M.$$  However, by the preceding Proposition, $(\inf_x \varphi(x,b))^M=(\inf_x \varphi(x,b))^{N'}$, whence $(\inf_x \varphi(x,b))^M=(\inf_x \varphi(x,b))^N$.  
\end{proof}

\begin{prop}[Uniform Continuity of Forcing]
For any $\la$-formula $\sigma(x)$ and any $\epsilon>0$, there is $\delta>0$ such that, for any $M\in \k$ and any tuples $a,a'\in M$, if $d(a,a')<\delta$, then $|V^M(\sigma(a))-V^M(\sigma(a'))|,|v^M(\sigma(a))-v^M(\sigma(a'))|<\epsilon$.
\end{prop}

\begin{proof}
By induction on the complexity of $\sigma$, the case of quantifier-free $\sigma$ being trivial.  Suppose that $\sigma(x)=\inf_y\varphi(x,y)$.  Let $\Delta^f_\varphi$ be a modulus of uniform continuity for forcing for $\varphi$.  Suppose that $M\in \k$, $a,a'\in M$ are within $\Delta^f_\varphi(\epsilon)$ and $M\nVdash \sigma(a)\geq r$.  Then there is $N\in \k$, $M\subseteq N$, and $b\in N$ such that $N\Vdash \varphi(a,b)<r$.  By definition of $\Delta^f_\varphi$, we have $N\Vdash \varphi(a',b)<r+\epsilon$, so $M\nVdash \sigma(a)\geq r+\epsilon$.  By symmetry, it follows that $|v^M(\sigma(a)-v^M(\sigma(a'))|<\epsilon$.  The other proofs are similar.    
\end{proof}

\begin{prop}
$\G$ is an inductive class.
\end{prop}

\begin{proof}
Suppose that $(M_\alpha \  : \ \alpha<\lambda\}$ is a chain from $\G$ and $M=\overline{\bigcup_{\alpha<\lambda}M_\alpha}$.  Since $\k$ is inductive, we have $M\in \k$.  Now suppose that $\sigma(a)$ is a restricted $\la_M$ sentence.  Fix $\epsilon>0$.  Choose $\delta>0$ to witness uniform continuity of forcing for $\sigma(x)$ and $\epsilon$.  Take $\alpha<\lambda$ and $a'\in M_\alpha$ such that $d(a,a')<\delta$.  Then:
$$V^M(\sigma(a))-v^M(\sigma(a))\leq 2\epsilon+(V^M(\sigma(a'))-v^M(\sigma(a')))=2\epsilon,$$ where the last equality holds as $M_\alpha$ is generic for $\sigma(a')$.  Let $\epsilon$ go to $0$.
\end{proof}

\begin{prop}\label{char}
Suppose that $\C$ is a subclass of $\k$ such that:
\begin{itemize}
\item $\C$ is model-consistent with $\k$, and
\item $\C$ is model-complete.
\end{itemize}
Then $\C\subseteq \G$.
\end{prop}

\begin{proof}
Fix $M_0\in \C$; we want $M_0\in \G$.  We prove by induction on complexity of $\sigma$ that $M_0$ is generic for $\sigma$.  Suppose first that $\sigma=\inf_x\varphi(x)$; we want $M_0$ generic for $\sigma$.  Suppose this is not the case and take $r\in(v^{M_0}(\sigma),V^{M_0}(\sigma))$.  Since $M_0\nVdash \sigma\geq r$, there is $N\in \k$ with $N\supseteq M_0$ and $a\in N$ such that $N\Vdash \varphi(a)<r$.  Let $M_1\in \C$ contain $N$, so $M_1\Vdash \varphi(a)<r$.  Since $M_1$ is generic, $\varphi^{M_1}(a)<r$, whence $\sigma^{M_1}<r$.  Let $M_2\in \G$ contain $M_1$ and $M_3\in \C$ contain $M_2$ and so on...  Let $M$ denote the union of the chain.  Since both $\C$ and $\G$ are model-complete classes, each $M_i$ is an elementary substructure of $M$.  In particular, $\sigma^{M_0}=\sigma^M=\sigma^{M_1}<r$, whence there is $b\in M_0$ such that $\varphi^{M_0}(b)<r$.  Since we already know that $M_0$ is generic for $\varphi(b)$, we have that $M_0\Vdash \varphi(b)<r$, whence $M_0\Vdash \sigma<r$, a contradiction.  The proof is similar for $\sigma=\sup_x\varphi(x)$.
\end{proof}

\begin{cor}
$\G$ is the unique maximal subclass of $\k$ that is model-consistent with $\k$ and model-complete.
\end{cor}

\begin{prop}
Suppose that $M\in \k$ is such that $M\preceq M'$ for all $M'\in \G$ with $M\subseteq M'$.  Then $M\in\G$.
\end{prop}

\begin{proof}
This follows from the Proposition \ref{char} by considering the class $\C:=\G\cup\{M\}$.
\end{proof}

\begin{prop}
Suppose that $\k=\operatorname{Mod}(T)$ for some $\forall\exists$-axiomatizable theory $T$.  Suppose that $M\in \k$ and $M'\in \G$ are such that $M\preceq M'$.  Then $M\in \G$.
\end{prop}

\begin{proof}
By the previous proposition, it is enough to show that if $N\in \G$ also contains $M$, then $M\preceq N$.  Since $M$ is an e.c. model of $T$ (being an elementary substructure of an e.c. model of $T$), we can find $N'\in \k$ which amalgamates $M'$ and $N$ over $M$.  Fix $N''\in \G$ extending $N'$.  Then for any $\la_M$ sentence $\sigma$, we have $$\sigma^M=\sigma^{M'}=\sigma^{N''}=\sigma^N.\qedhere$$
\end{proof}

\begin{cor}
Suppose that $\k=\operatorname{Mod}(T)$ for some $\forall\exists$-axiomatizable theory $T$.  Then for every $M\in \k$, there is $N\in \G$ with $M\subseteq N$ and such that the density character of $N$ equals the density character of $M$. 
\end{cor}

\begin{proof}
This is immediately from the previous proposition, Proposition \ref{contain}, and Downward L\"owenheim-Skolem.
\end{proof}

%
%

Let $T^g:=\Th(\G)$.  We call $T^g$ the \emph{forcing companion} for $T$.  The name is a good one:

\begin{prop}
$T^g$ is a companion operator.
\end{prop}

\begin{proof}
Immediate from the fact that every element of $\G$ is existentially closed.
\end{proof}

Recall that a theory $T$ has the \emph{joint embedding property} (JEP) if any two models of $T$ can be simultaneously embedded into a third model of $T$.  If $T$ is complete, then $T$ has JEP.

\begin{lemma}
$T^g$ is complete if and only if $T$ has JEP.
\end{lemma}

\begin{proof}
If $T^g$ is complete, then $T^g$ has JEP; since $T$ and $T^g$ have the same universal theories, it follows that $T$ has JEP.

Conversely, suppose that $T$ has JEP (whence it follows that $T^g$ has JEP); we must show that $T^g$ is complete.  Suppose $\m,\N\in \G$.  Let $\A\models T^g$ be such that $\m,\N$ both embed into $\A$.  Let $\A_1\in \G$ be such that $\A$ embeds into $\A_1$.  From model-completeness of $\G$, we see that $\m,\N\preceq \A_1$, whence $\m\equiv \N$.  It follows that $T^g$ is complete.
\end{proof}

%
%

We are now ready to show that the infinite forcing companion of $\Th_\forall(\R)$ is $\Th(\R)$.
\begin{prop}\label{Rgeneric}
$\R$ is infinitely generic.
\end{prop}

\begin{proof}
We prove, by induction on complexity of restricted $L(\R)$-sentences $\sigma$, that $v^\R(\sigma)=V^\R(\sigma)=\sigma^\R$.  This is clear for $\sigma$ quantifier-free.  Now suppose that $\sigma=\inf_x\varphi(x,a)$, where we display the parameters $a$ coming from $\R$.  We first prove that $V^\R(\sigma)\leq \sigma^\R$.  Suppose that $\sigma^\R<r$, so $\varphi(b,a)^\R<r$ for some $b\in \R$.  By the induction hypothesis, $\R\Vdash \varphi(b,a)<r$, so $\R\Vdash \sigma<r$ and $V^\R(\sigma)\leq r$.  We now prove that $\sigma^\R\leq v^\R(\sigma)$.  Suppose that $v^\R(\sigma)<r$, so $\R\nVdash \sigma \geq r$.  Then there is $N\supseteq \R$ and $b\in N$ such that $N\Vdash \varphi(b,a)<r$.  Let $N_1\supseteq N$ be infinitely generic.  Then
$$\varphi(b,a)^{N_1}= V^{N_1}(\varphi(b,a))\leq V^N(\varphi(b,a))<r.$$  Let $N_2\preceq N_1$ contain $\R$ and $b$.  Let $j:N_2\to \R^\u$ be an embedding.  Then $\varphi(j(b),j(a))^{\R^\u}<r$ whence $\inf_x\varphi(x,j(a))^{\R^\u}<r$.  Since the induced embedding $\R\hookrightarrow N_2\hookrightarrow \R^\u$ is elementary, we have $\sigma^\R=\inf_x\varphi(x,a)^\R<r$.   

The case that $\sigma=\sup_x\varphi(x,a)$ is similar, but we include a proof for the sake of completeness.  We first prove that $\sigma^\R\leq v^\R(\sigma)$.  Suppose that $v^\R(\sigma)<r$.  Then $\R\nVdash \sigma>r$, that is, $\R\nVdash \varphi(b,a)>r$ for all $b\in \R$, that is, $v^\R(\varphi(b,a))\leq r$.  By the induction hypothesis, $\varphi(b,a)^\R \leq r$ for all $b\in \R$, whence $\sigma^\R\leq r$.  We now show that $V^\R(\sigma)\leq \sigma^\R$.  Suppose that $V^\R(\sigma)> r$.  Then $\R\nVdash \sigma\leq r$.  Thus, there is $N\supseteq \R$ and $b\in N$ such that $N\Vdash \varphi(b,a)>r$.  Let $N_1\supseteq N$ be infinitely generic.  Then
$$\varphi(b,a)^{N_1}=v^{N_1}(\varphi(b,a))\geq v^N(\varphi(b,a))\geq r.$$  Let $N_2\preceq N_1$ contain $\R$ and $b$.  Let $j:N_2\to \R^\u$ be an embedding.  Then $\varphi(j(b),j(a))^{\R^\u}\geq r$, whence $\sup_x\varphi(x,j(a))^{\R^\u}\geq r$.  As before, that means that $\sigma^\R=\sup_x\varphi(x,a)^\R\geq r$.
\end{proof}

\begin{cor}
If $T=\Th_\forall(\R)$, then $T^g=\Th(\R)$.
\end{cor}

\begin{proof}
Since $T$ has JEP, $T^g$ is complete; since $\R\models T^g$, it follows that $T^g=\Th(\R)$.
\end{proof}

\begin{cor}
Every $\R^\omega$ embeddable II$_1$ factor is contained in an e.c. model of $\Th(\R)$.
\end{cor}

\begin{proof}
Let $M$ be an $\R^\omega$-embeddable II$_1$ factor.  Then there is infinitely generic $N$ with $M\subseteq N$.  But since $T^g=\Th(\R)$, it follows that $N\equiv \R$.
\end{proof}

If we knew that $\Th(\R)$ were $\forall\exists$-axiomatizable, then the previous corollary would be immediate.  

\begin{cor}
There are continuum many nonisomorphic e.c. models of $\Th(\R)$.
\end{cor}

\begin{proof}
Combine the previous corollary with Fact \ref{NPS}.
\end{proof}
\section{Finitely generic structures}\label{fingenericsection}

There is another kind of model-theoretic forcing that is more in the spirit of Cohen's original notion of forcing which is often called \emph{finite (model-theoretic) forcing}.  This forcing was adapted to the continuous setting in \cite{BI} and we only recall the basic setup in order to give context to our results.

We work in a countable signature $\la$ and add countably many  new constant symbols $C$ to the language.  We fix a class $\k$ of structures and let $\k(C)$ denote the class of all structures $(M,a_c)_{c\in C_0}$, where $C_0$ is a finite subset of $C$.  We treat such structures as $\la(C_0)$-structures in the natural way.

Conditions are finite sets of the form $\{\varphi_1<r_1,\ldots,\varphi_n<r_n\}$, where each $\varphi_i$ is an atomic $\la(C)$-sentence and such that there is $M\in \k(C)$ such that $\varphi_i^M<r_i$ for each $i=1,\ldots,n$.  The partial order on conditions is reverse inclusion.  If $p$ is a condition and $\varphi$ is an atomic sentence of $\la(C)$, we define $f_p(\varphi):=\min\{r\leq 1 \ | \varphi<r\in p\}$, with the understanding that $\min(\emptyset)=1$.  For a condition $p$ and an $\la(C)$-sentence $\varphi$, we define the value $F_p(\varphi)\in [0,1]$ by induction on $\varphi$.
\begin{itemize}
\item $F_p(\varphi)=f_p(\varphi)$ if $\varphi$ is atomic.
\item $F_p(\neg \varphi)=\neg \inf_{q\supseteq p}F_q(\varphi)$.
\item $F_p(\frac{1}{2}\varphi)=\frac{1}{2}(\varphi)$.
\item $F_p(\varphi+\psi)=F_p(\varphi)+F_p(\psi)$.  (Truncated addition)
\item $F_p(\inf_x\varphi(x))=\inf_{c\in C}F_p(\varphi(c))$.
\end{itemize}

If $r\in \r$ and $F_p(\varphi)<r$, we say that $p$ \emph{forces} that $\varphi<r$, and write $p\Vdash \varphi<r$.

\begin{df}
We say that a nonempty set $G$ of conditions is \emph{generic} if the union of two elements of $G$ is once again an element of $G$ and for every $\la(C)$-sentence $\varphi$ and every $r>1$, there is $p\in G$ such that $F_p(\varphi)+F_p(\neg \varphi)<r$. 
\end{df}

If $G$ is generic and $\varphi$ is an $\la(C)$-sentence, set $\varphi^G:=\inf_{p\in G}F_p(\varphi)$.  We should also say that generic sets exist; in fact, any condition is contained in a generic set by \cite[Proposition 2.12]{BI}.

The following result is the combination of Lemma 2.16 and Theorem 2.17 in \cite{BI}.
\begin{thm}[Generic Model Theorem]
Let $M_0^G$ denote the term algebra $\mathcal T(C)$ equipped with the natural interpretation of the function symbols and interpreting the predicate symbols by $P^{M_0^G}(\vec \tau):=P(\vec \tau)^G$.  Let $M^G$ be the completion of $M_0^G$.  Then $M^G$ is an $\la(C)$-structure such that, for all $\la(C)$-sentences $\varphi$, we have $\varphi^{M^G}=\varphi^G$.
\end{thm}

We say that an $\la$-structure $N$ is \emph{finitely generic for $\k$} if there is a generic $G$ such that $M$ is isomorphic to the $\la$-reduct of $M^G$.  Note that finitely generic structures exist as generic sets of conditions exist.  Finitely generic structures are existentially closed.  

Let $T^f$ denote the theory of the class of finitely generic models of $T$.  Then $T^f$ is a companion operator for $T$, called the \emph{finite forcing companion}, and is complete if and only if $T$ has JEP.  (See Chapter 5 of \cite{Hirsh} for the proofs of these claims in the classical case.)

\begin{prop}\label{ecgeneric}
Suppose that $M,N\in \k$, $M\subseteq N$, and $M$ is existentially closed in $N$.  If $N$ is finitely generic for $\k$, then $M$ is finitely generic for $\k$.
\end{prop}

\begin{proof}
See \cite[Proposition 5.15]{Hirsh} for a proof in the classical case. 
\end{proof}

\begin{cor}
$\R$ is finitely generic and $\Th(\R)$ is the finite forcing companion of $\Th_\forall(\R)$.
\end{cor}

\begin{proof}
Suppose that $M$ is a finitely generic model of $\Th_\forall(\R)$.  Then $M$ is e.c., whence a II$_1$ factor.  Since $\R$ embeds into $M$ and is e.c., we have that $\R$ is finitely generic.  The second claim follows from the fact that the forcing companion is complete whenever the original theory has JEP.
\end{proof}

Hodges' book \cite{H} describes the finitely generic models being at the ``thin'' end of the spectrum of e.c. models while the infinitely generic ones are at the ``fat'' end.  It is thus interesting that in the case of $\Th(\R)$, we have the prime model being both finitely and infinitely generic while simultaneously having a plethora of infinitely generic models (and yet not being model complete).

\end{document}